\documentclass[12pt]{amsart}
\usepackage{graphicx}
\usepackage{psfrag}
\usepackage{amssymb}
\usepackage{epstopdf}
\usepackage[all,arc,dvips]{xy}
\usepackage{a4wide}
\usepackage{enumerate}
\DeclareMathOperator{\Ext}{Ext}
\DeclareMathOperator{\Area}{Area}
\DeclareMathOperator{\Mod}{Mod}

\newcommand{\al}{\alpha} 
\newcommand{\be}{\beta}         
\newcommand{\de}{\delta}

\newcommand{\e}{\epsilon}

\newcommand{\ga}{\gamma}

\newcommand{\R}{\mathbb{R}}

\newcommand{\mc}{\mathcal}

\begin{document}

\title{Criteria for the Divergence of pairs of Teichm\"uller geodesics}
\newtheorem{lem}{\textbf{Lemma}}
\newtheorem{thm}{\textbf Theorem}
\newtheorem{cor}{\textbf Corollary}
\newtheorem{prp}{\textbf Proposition}
\newtheorem{df}{\textbf Definition}
\newtheorem{introthm}{\textbf{Theorem}}
\renewcommand{\theintrothm}{\Alph{introthm}}

\author{Anna Lenzhen and Howard Masur}
 \thanks{The second author is 
supported in part by the NSF}

\date{\today}

\begin{abstract}
We study the asymptotic geometry of Teichm\"uller geodesic rays.  We show that, when the transverse measures to the vertical foliations of the quadratic differentials determining two different  rays are topologically equivalent, but are not absolutely continuous with respect to each other,  the rays  diverge in Teichm\"uller space. 
\end{abstract}

\maketitle

\section{Introduction}
Let $S$ be an oriented surface of genus $g$ with $n$ punctures.  We assume $3g-3+n\geq 1$.  Let $\mc T(S)$ denote the Teichm\"uller space of $S$ with the Teichm\"uller metric $d(\cdot,\cdot)$.   A basic question in geometry is to study the long term behavor of geodesics.  In this paper we study the question of when 
a pair of geodesic rays $X_1(t),X_2(t)$, with possibly distinct basepoints, stay bounded distance apart, and when they diverge in the sense that $d(X_1(t),X_2(t))\to \infty$ as $t\to\infty$.  

 Teichm\"uller's theorem implies that a  Teichm\"uller geodesic ray is determined by a 
quadratic differential $q$  at the base point and that there are quadratic differentials $q(t)$ on $X(t)$ along the ray found by stretching along the horizontal trajectories of $q$ and contracting along the vertical trajectories. 

Many cases of the question of divergence of rays are already known.  It is a general principle that the asymptotic behavior of the ray is determined by the properties of the vertical foliation  of $q$.
The first instance is if the  quadratic differentials $q_1,q_2$ defining the geodesic rays $X_1(t),X_2(t)$  are Strebel differentials. This means that  their vertical trajectories are closed and   decompose the surface into cylinders.   In \cite{Ma:1} it was shown that  if the homotopy classes of the cylinders for $q_1$ coincide with those of $q_2$, then
$X_1(t),X_2(t)$ stay bounded distance apart. In particular this showed that
the Teichm\"uller metric was not negatively curved in the sense of
Busemann. A second known case is if the vertical foliations of
$q_1,q_2$ are the same uniquely ergodic foliation.
In that case the rays also stay bounded distance apart (\cite{Ma:2}).

The next possibility is  the vertical foliations of $q_1,q_2$ are topologically equivalent, have a minimal component and  yet are  not uniquely ergodic.  (It is well-known that for any quadratic differential, the vertical trajectories decompose the surface into cylinders and subsurfaces in which every trajectory is dense). In that case in each minimal component there exist a finite number of mutually singular ergodic measures, and any transverse measure is a convex combination of the ergodic measures.   Ivanov (\cite{I}) showed that if the transverse measures of $q_1,q_2$ in these minimal components are absolutely continuous with respect to each other, then the rays stay bounded distance apart. 
In this paper we prove the converse.

Let $q_1,q_2$ be quadratic differentials on $X_1$ and $X_2$ with vertical foliations 
$[F_{q_1}^v,|dx_1|]$ and $[F_{q_2}^v,|dx_2|]$ and  determining rays $X_1(t),X_2(t)$.
Our main result is then 
\begin{introthm}\label{main}
Suppose $F_{q_1}^v$ and $F_{q_2}^v$ are topologically equivalent. Suppose there is a   minimal  component $\Omega$ of the foliations $F_{q_i}^v$  with ergodic measures $\nu_1,\ldots, \nu_p$ and so that restricted to $\Omega$,  $|dx_1|=\sum_{i=1}^p a_i\nu_i$, $|dx_2|=\sum_{i=1}^p b_i\nu_i$ and there is some index $i$ so that either $a_i=0$ and $b_i>0$, or $a_i>0$ and $b_i=0$.  Then the rays $X_1(t)$ and $X_2(t)$ diverge. 
\end{introthm}
  In particular, this holds when the transverse measures are distinct ergodic measures.

The last possibility is that the 
vertical foliations
of $q_1,q_2$ are not topologically equivalent.  If the
geometric intersection of the vertical foliations is nonzero,   then the rays diverge (\cite{I}). 
We prove 
\begin{introthm}
\label{top}
Suppose $q_1,q_2$ are quadratic differentials such that the vertical
foliations  $[F_{q_1}^v,|dx_1|]$ and $[F_{q_2}^v,|dx_2|]$   are not
topologically equivalent, but $i([F_{q_1}^v,|dx_1|],
[F_{q_2}^v,|dx_2|])=0$. Then the rays $X_1(t))$ and $X_2(t)$ diverge. 
\end{introthm}


These theorems together with the previously known results completely
answer the question of divergence of rays.

The outline of the proof of  Theorem~\ref{main} is as follows. In Proposition~\ref{subsurface}, we will show that,  for the flat metrics defined by the quadratic differentials  $q_1(t)$ and $q_2(t)$, for any sufficiently large time $t$, there is a subsurface  
$Y(t)\subset \Omega$ with its area small in one metric and bounded
away from zero in the other, while its boundary is short in both
metrics. This is where we use the assumption that the measures are not
absolutely continuous with respect to each other. We will then apply
Lemma~\ref{vertical} to find a bounded length curve $\gamma(t)\subset
Y(t)$ which is "mostly vertical" with respect to the metric of
$q_1(t)$.  It has  comparable length in the
metric of $q_2(t)$.   Using  the fact that the quadratic differentials give comparable length to $\gamma(t)$ while giving very different areas to $Y(t)$,  Lemma~\ref{area} will allow us  to show that the ratio of the extremal length of $\gamma(t)$  along one ray to the extremal length on the other is large. We then apply Kerckhoff's formula to conclude that the surfaces are far apart in Teichmuller space.

We will also prove 
\begin{introthm}
\label{approx}
 Let $\nu_1,\ldots, \nu_p$ be maximal collection of ergodic measures for a minimal foliation $[F,\mu]$.   Then there is a sequence of multicurves  $\gamma_n=\{\gamma^1_n,\ldots, \gamma^k_n\}$ such that $\gamma^j_n\to [F,\nu_j]$ in $\mc {PMF}$.   
\end{introthm}
In other words, any two topologically equivalent measured foliations can be approximated by a sequence of multicurves, with possibly different weights. 
  This result settles a question asked by Moon Duchin. 
        
{\bf Acknowledgements. }We would like to thank Moon Duchin,   
Ursula Hamenstadt, and Kasra Rafi   for useful discussions and their interest in this project.  We would also like to thank the referee for numerous helpful suggestions.  We are also grateful to the Mathematical Sciences Research Institute for the support and hospitality during the time this research was conducted.

\section{Background}
\subsection{Measured foliations}
Recall a measured foliation on a surface $S$  consists of  a finite set $\Sigma$ of singular points and a covering of $S\setminus \Sigma$  by open sets  $\{U_\alpha\}$ with charts $\phi_\alpha:U_\alpha\to \R^2$ such that the overlap maps are of the form 
$$(x,y)\to (\pm x+ c,f(x,y)).$$  The leaves of the foliation 
are the lines $x=\text{constant}$. The points $\Sigma$ are $p$-pronged singularities for $p\geq 3$. One allows  single  pronged singularities at the punctures. 
 A measured foliation comes equipped with a transverse invariant measure which in the above coordinates is given by $\mu=|dx|$.  Henceforth we will denote measured foliations by $[F,\mu]$. We will write $F$ to denote a (topological) foliation when we are ignoring the measure.  

For the rest of the paper, a curve  will always mean a simple closed
curve. 
For any homotopy class of simple closed curves $\beta$, let $$i([F,\mu],\beta)=\inf_{\beta'\sim\beta}\int_{\beta'} d\mu.$$
The intersection  with simple closed curves  extends to an
intersection 
function  $$i([F_1,\mu_1],[F_2,\mu_2])$$ on pairs of measured
foliations. 
Thurston's space of measured foliations is denoted $\mc MF$ and the projective space of measured foliations by $\mc PMF$.

Let $\Gamma_F$ denote the compact leaves of $F$ joining singularities. 
It is well-known that
 each component  $\Omega$ of $S\setminus \Gamma_F$ is either 
 an annulus  swept out by closed leaves 
or  a minimal domain in which  every leaf  is dense.

\begin{df}
We say that two foliations $F_1,F_2$ on $S$ are topologically equivalent, and
we write $F_1\sim F_2$, if there is a homeomorphism of $S\setminus \Gamma_{F_1}\to S\setminus \Gamma_{F_2}$ isotopic to the identity which takes the leaves of $F_1$ to the leaves of $F_2$.
\end{df}

 Note that this definition does not refer to the measures. 
\begin{df}
A foliation $[F,\mu]$ in a minimal domain $\Omega$ is  said to be {\em uniquely ergodic} if the   transverse measure $\mu$ restricted to  $\Omega$  is the unique transverse measure of the foliation $F$ up to scalar multiplication. 
\end{df}

More generally, suppose 
$\Omega$ is a minimal component of $[F,\mu]$.  There exist invariant transverse measures
$\nu_1=\nu_1(\Omega),\ldots, \nu_p=\nu_p(\Omega)$  such that 
\begin{itemize}
\item $p$ is bounded in terms of the topology of $\Omega$.
\item $\nu_i$ is ergodic for each $i$.
\item any transverse invariant  measure $\nu$ on $\Omega$ can be written as 
$\nu=\sum_{i=1}^p a_i\nu_i$ for $a_i\geq 0$.

\end{itemize}
Thus the transverse measures are parametrized by a simplex in $\R^p$.
Two foliations, $[F,\mu_1]$ and $[F,\mu_2]$, are absolutely continuous with respect to each other if, when the measures are expressed as a convex combination as above, the indices with positive coefficients are identical. 
Equivalently, they are absolutely continuous with respect to each other if they lie in the same open face of the simplex. 

\subsection{Quadratic differentials and Teichm\"uller rays}
A meromorphic quadratic differential $q$ on a closed Riemann surface $X$ with a finite number of punctures removed is a tensor of the form $q(z)dz^2$ where $q$ is a holomorphic function and $q(z)dz^2$ is invariant under change of coordinates. We allow $q$ to have at most simple poles at the punctures. 
 
As such there is a metric defined by $|q(z)|^{1/2}|dz|$.  
The length of an arc $\beta$ with respect to the metric will be denoted by $|\beta|_q$. There is an area element defined by $|q(z)||dz|^2$. 
We will denote by $\Area_q \Omega$ the area of a subsurface $\Omega\subseteq X$. 

 Away from the zeroes and poles of $q$ there are {\em natural} holomorphic coordinates $z=x+iy$ such that in these coordinates $q=dz^2$.  
 The lines  
$x=\text{constant}$ with transverse measure $|dx|$ define the vertical foliation $[F_q^v,|dx|]$. 
The  lines  $y=\text{constant}$ with transverse measure $|dy|$ define the horizontal measured foliation $[F_q^h,|dy|]$. 
The transverse measure  of an arc $\beta$ with respect to $|dy|$ will be denoted by $v_q(\beta)$ and called the vertical length of $\beta$.  Similarly, we have the  horizontal  length denoted by 
$h_q(\beta)$. The area element in the natural coordinates is given by $dxdy$.

We denote by $\Gamma_q=\Gamma_{F_q^v}$ the vertical critical graph of $q$.
This is  the union of the vertical leaves joining the zeroes of $q$. 

The 
{\em Teichm\"uller space} of
$S$ denoted by  $\mc T(S)$ is the set of equivalence classes of Riemann surface structures $X$ on $S$,  where 
$X_1$ is equivalent to $X_2$ if there is a conformal map $f:X_1\to X_2$ isotopic to the identity on $S$. 
The {\em Teichm\"{u}ller metric} on $\mc T(S)$ is the metric
defined by $$d(X_1,X_2):=\frac{1}{2}\inf_f\{\log
K(f): f:X_1\to X_2 \mbox{\ is homotopic to Id\ }\}$$ where
$f$ is quasiconformal and $$ K(f):=||K_x(f)||_\infty\geq 1 $$
\noindent is the {\em quasiconformal dilatation} of $f$, where
$$K_x(f):=
\frac{|f_z(x)|+|f_{\bar{z}}(x)|}{|f_z(x)|-|f_{\bar{z}}(x)|} $$
\noindent is the {\em pointwise quasiconformal dilatation} at
$x$.

{\em Teichm\"{u}ller's Theorem}  states that, given any $X_1,X_2\in\mc T(S)$, there exists a unique 
(up to translation in the case when $S$ is a torus) 
quasiconformal map $f$, called the {\em Teichm\"{u}ller map}, realizing $d(X_1,X_2)$.   
The Beltrami coefficient $\mu_f:=\frac{\bar{\partial }f}{\partial f}$ is of the form $\mu_f=k\frac{\bar{q}}{|q|}$ for 
a unique unit area quadratic differential $q$ on $X_1$ and some $k$ with $0\leq k<1$.  Define $t$ by $$e^{2t}=\frac{1+k}{1-k}.$$
There is a quadratic differential $q(t)$ on $X_2$ such that in the natural 
local coordinates $w=u+iv$ of $q(t)$ and $z=x+iy$ of $q$ the map $f$ is given by $$u=e^tx\ \ v=e^{-t}y.$$
Thus $f$ expands along the horizontal leaves of $q$  by $e^t$, and contracts along the vertical leaves by $e^{-t}$.   

Conversely, any unit area $q$ on $X$ determines a $1$-parameter family of Teich\"muller maps $f_t$ defined on $X$. Namely $f_t$ has Beltrami differential   $\mu=k\frac{\bar{q}}{|q|}$ where $e^{2t}=\frac{1+k}{1-k}$.  The image surface is denoted by $X(t)$ and $X(t); t\geq 0$ is the {\em Teichm\"{u}ller ray} based at $X$ in the direction of $q$.  On each $X(t)$ we have the quadratic differential $q(t)$. 



\subsection{Extremal length and Annuli}
We recall the notion of extremal length. 
Suppose $X$ is a Riemann surface and $\Gamma$ is a family of arcs on $X$.  Suppose $\rho$ is a conformal metric on $X$.  For an arc $\gamma$, denote by $\rho(\gamma)$ its length and by $A(\rho)$ the area of $\rho$.
\begin{df} $$\Ext_X(\Gamma)=\sup_\rho \frac{\inf_{\gamma\in\Gamma}\rho^2(\gamma)}{A(\rho)},$$ where the sup is over all conformal metrics $\rho$.
\end{df}

We will apply this definition when $\Gamma$ consists of all simple closed curves in a free homotopy class of some $\alpha$. In that case we will write $\Ext_X(\alpha)$. 
It is also worth noting that if $q$ is a unit area quadratic
differential then  $\Ext_X(\alpha)\geq |\alpha|_q^2$  since $q$ gives
a competing metric. (Here again $|\alpha|_q$ denotes the length of the geodesic in the homotopy class of $\alpha$.)

The following formula due to Kerckhoff (\cite{K:1}) is extremely useful in estimating Teichm\"uller distances.
For $X_1,X_2\in\mc T(S)$
\begin{equation}
\label{eq:Kerck}
d(X_1,X_2)=\frac{1}{2}\log\sup_\alpha \frac{\Ext_{X_2}(\alpha)}{\Ext_{X_1}(\alpha)}.
\end{equation}

We will also need the following inequalities, comparing hyperbolic and extremal lengths.  They are given by  Corollary 3 in Maskit\cite{Mas:1}. The first says that 
$$\Ext_X(\alpha)\leq \frac{1}{2}l_\sigma(\alpha)e^{\frac{1}{2}l_\sigma(\alpha)}$$
where $l_{\sigma}(\alpha)$ repsesents the length of the geodesic, in the homotopy class of $\alpha$,  with respect to the hyperbolic metric of $X$. 
The second says that as $l_\sigma(\alpha)\to 0$ we have $$\frac{\Ext_X(\alpha)}{l_\sigma(\alpha)}\to 1.$$

\begin{df}
Suppose there is an embedding of a Euclidean cylinder in $\R^3$ into $X$ which is an isometry with respect to the metric of $q$.  The image is called a flat cylinder.  The cylinder is maximal if it cannot be enlarged. In that case there are singularities on each boundary component of the cylinder.    
\end{df}

We need a definition and estimates found in \cite{CRS}  and \cite{Min:2}.
We first adopt the following notation. If two quantites $a$ and $b$ differ by multiplicative and additive constants that depend only on the topology, then we will often write 
$$a\asymp b.$$

\begin{df} Given a quadratic differential with its metric $q$, an expanding annulus $A$ is an annulus where the curvature of each boundary component has constant sign, either positive or negative at each point, the boundary curves are equidistant and there are no zeroes inside $A$. 
\end{df}
 Let $d(A)$ be the distance between the boundary components of an expanding annulus. 
It is universally bounded. 
 The following statement can be found as Corollary 5.4 of \cite{CRS}.   
\begin{lem}
\label{lem:ext}
Suppose $q$ is a quadratic differential of area $1$ on $X$ with its hyperbolic metric $\sigma$,  and  $\beta$ is a sufficiently short curve.
\begin{enumerate}
\item   
$$\frac{1}{l_\sigma(\beta)}\asymp \max(\Mod(F(\beta),\Mod(A(\beta)),$$ where $F(\beta)$ is the maximal flat cylinder, $A(\beta)$ is the maximal expanding annulus with one boundary component the $q$- geodesic in the class of $\beta$ and 
\item  
$$\Mod(A(\beta))\asymp \log\frac{d(A(\beta))}{|\beta|_q}.$$   
\item the other boundary component of $A(\beta)$ contains a zero of $q$. 
\end{enumerate}
\end{lem}

 We  will also need the following result from  \cite{Min:1}. Minsky gives a useful estimate of the extremal length of a curve, which uses {\sl collar decomposition}. For  $0<\e_1<\e_0$ less than the Margulis constant, let $\mc A $ be the collection of pairwise disjoint annular neighborhoods of the geodesics  of hyperbolic length at most $\e_1$, whose internal boundary components have hyperbolic length $\e_0$.  Then  the union of $\mc A$  with the collection of components of $X- \mc A$  is the $(\e_0,\e_1)$ collar decomposition of $X$. For a subsurface $Q\subset X$, and $\alpha$ a homotopy class of curves, we denote by $\Ext_Q(\alpha)$ the extremal length of the restriction of the curves in  $\alpha$ to $Q$.

   \begin{thm}\label{minsky}
[Theorem 5.1 in \cite{Min:1}] Let $X$ be a Riemann surface of finite type with boundary lengths in the hyperbolic metric at most $\ell_0$, and let $\mc Q$ be the set of components of the $(\e_0,\e_1)$ collar decomposition of X. Then, for any curve $\alpha$ in $X$, then  $$\Ext_X(\alpha)\asymp  \underset{Q\in \mc Q}{\max} \Ext_{Q}(\alpha)$$
  where the multiplicative factors depends only on $\e_0,\e_1,\ell_0$ and the topological type of $X$.\label{minsky}
  \end{thm}

\subsection{Limits of quadratic differentials}
We need the following  convergence result for quadratic differentials where one or more curves have extremal length approaching $0$. The proof follows more or less immediately from results in  \cite{Ma:5}.  
\begin{thm}
\label{thm:limits}
Suppose $X_n$ is a sequence of Riemann surfaces, $q_n$ is a sequence
of unit area quadratic differentials on $X_n$, 
and $\gamma_1,\ldots, \gamma_j$ is a  collection of disjoint simple closed curves such that 
\begin{itemize}
\item the  extremal length of each $\gamma_i$ goes to $0$ along $X_n$
\item the extremal length of every other closed curve is bounded below away from $0$ along the sequence. 
\item there is no flat cylinder in the homotopy class of $\gamma_i$ 
\end{itemize}
Then by passing to a subsequence, for any subsurface $\Omega_n\subset X_n$ bounded by the geodesic representatives of the $\gamma_i$,  whose $q_n$-area is bounded away from $0$, there is a surface $\Omega_\infty$ with punctures and a nonzero finite area quadratic differential $q_\infty$ on $\Omega_\infty$ such that $q_n$ restricted to $\Omega_n$ converges uniformly on compact sets to $q_\infty$. 

The convergence means that for any  neighborhood $U$ 
of the punctures on $\Omega_\infty$  
\begin{enumerate}
\item \label{conv1} for large enough $n$ there is a conformal map $F_n:\Omega_\infty\setminus U\to X_n$
\item \label{conv2} $F_n^*q_n\to q_\infty$ as $n\to\infty$ uniformly on $\Omega_\infty\setminus U$.
\end{enumerate} 
\end{thm}

\begin{proof}  Using the compactification of the moduli space of Riemann surfaces (see \cite{Bers:1}), by passing to a subsequence we can assume $X_n$
converges to a limiting Riemann surface $X_\infty$ with paired punctures corresponding to each $\gamma_i$  so that \eqref{conv1} holds above.  Then again passing to a subsequence we can assume  $q_n$ converges to some finite area quadratic differential $q_\infty$ on each component $\Omega_\infty$ of $X_\infty$; the convergence  as in \eqref{conv2}.  We need to show that that if $\Omega_n$ has $q_n$-area bounded below then $q_\infty$ is not identically $0$ on the corresponding $\Omega_\infty$. For each paired punctures on $X_\infty$ pick holomorphic coordinates  $0<|z_i|<1$  and $0<|w_i|<1$ 
on the corresponding punctured discs.  For $n$ sufficiently large, for each $i$ there is a $t_i=t_i(n)$ which goes to $0$ as $n\to\infty$  such that $X_n$ can be recovered from $X_\infty$ by removing the punctured discs $0<|z_i|<|t_i|$ and $0<|w_i|<|t_i|$ and then gluing the annulus $|t_i|\leq |z_i|\leq 1$ to the annulus $|t_i|\leq |w_i|\leq 1$ by the formula 
$$z_iw_i=t_i.$$ This produces  an annulus in the class of $\gamma_i$.  In forming $X_n$, we also allow a small deformation of the complex structure of $X_\infty$ in the complement of the union of the discs.  We need to consider the punctured discs $0<|z_i|<1$ contained in $\Omega_\infty$ and the corresponding annulus  
$$A_i=\{z_i:|t_i|^{1/2}<|z_i|<1\}\subset X_n.$$  In the coordinates of $A_i$ the $q_n$-geodesic in the class of $\gamma_i$ lies outside any fixed compact set $K$ for $n$ large enough. Fix now the  index $i$ and suppress that subscript. By the Corollary following Lemma 5.1 in \cite{Ma:5}, we can express  $q_n$  in $A=A_i$ as 
$$q_n=a_n/z^2+f_n/z+tg_n/z^3,$$
where $a_n,f_n,g_n$ are uniformly bounded family of  holomorphic functions of $z$.  
 It is easy to see that the last term integrated over $A$ goes to $0$ as $t=t_i$ goes to $0$.  Since there is no flat annulus in the class of $\gamma_i$,
by Lemma 5.3 of \cite{Ma:5}, we have $$-|a_n|^{1/2}\log |t|\leq 1.$$ This implies that the first term of the expansion of $q_n$ also has small integral over $A$. Since we are assuming that the integral of $|q_n|$  is bounded away from $0$ on $\Omega_n$ we must have that $f_n$ converges to a nonzero function on the disc $0<|z|<1$, and so $q_\infty$ is not identically $0$. 
\end{proof}

\section{Lemmas relating length, slope and area}

We need to recognize instances when the area of a subsurface is small.  
As a consequence of the preceeding Theorem  we show that if all bounded length curves have small horizontal length, then the area  is small. 

\begin{lem}
\label{lem:everyvert}
With the same assumption as in Theorem~\ref{thm:limits}  suppose $q_n'$ is another quadratic differential on $X_n$ such that $h_{q_n'}(\alpha)\to 0$ for any fixed homotopy class of curves in a nonannular component $\Omega_n$ of the  complement of the curves $\gamma_1,\ldots, \gamma_p$. 
Then $\Area_{q_n'}(\Omega_n)\to 0$.

\end{lem}
\begin{proof}
By passing to a subsequence we can assume $\Omega_n\to\Omega_\infty$ and $q_n'\to q'_\infty$.  Now each geodesic  $\alpha$ of $q'_\infty$ has  horizontal length equal to $0$ which is impossible, since $\Omega_\infty$ is not a flat cylinder.  
\end{proof}

The next  lemma compares areas of flat cylinders with respect to different flat metrics. 

\begin{lem}
\label{lem:annulus}
Suppose $q_1,q_2$ are quadratic differentials with the same horizontal foliation $|dy|$  and whose vertical foliations are topologically equivalent with transverse measures $\nu_1,\nu_2$.  For any $B>0$, there exists $\epsilon_0,M$ such that for all $\epsilon<\epsilon_0$, if  $C_1=C_1(\beta)$ is a maximal flat cylinder for $q_1$ with core curve $\beta$ with the properties that  
\begin{itemize}
\item  The absolute value of the slope of $\beta$ in $C_1$ is at least $1$.
\item  $|\beta|_{q_1}\leq \epsilon_0$.
\item $\text{Area}_{q_1}(C_1)\geq B$. 
\item Any horizontal segment  $I$ crossing $C_1$ satisfies  $\nu_2(I)\leq \epsilon$ 
\end{itemize}
then, if  $C_2$ is the maximal flat cylinder defined by $q_2$ in the class of
 $\beta$, we have 
$\text{Area}_{q_2} (C_2)\leq 2\epsilon\epsilon_0$.
Moreover the ratio of lengths of vertical arcs crossing the cylinders are comparable. 
\end{lem}
\begin{proof}
We may represent $C_1$ as a parallelogram with a pair of horizontal sides that are glued to each other by a translation. 
Let  $I$ be an oriented  horizontal segment crossing $C_1$ starting  at a singularity $P_0$ on one boundary component.  Let $Q_0$ be the  endpoint of $I$ on the other boundary component.  
Assume without loss of generality that 
the 
slope of $\beta$ in $C_1$ is negative.   This means that there is a vertical leaf
through $P_0$ that enters $C_1$ in the positive direction and returns to $I$ without leaving $C_1$ and translated by   
$h_1:=h_{q_1}(\beta)\leq \epsilon_0$.    Starting at $P_0$, for $\epsilon_0$ 
sufficiently small compared to $B$, there will be at least two additional  returns for the vertical leaf through $P_0$ before the leaf leaves the cylinder. Since the vertical foliations of $q_1$ and $q_2$ coincide, the same is true for the vertical leaf of  $q_2$ leaving $P_0$, although now the translation amount, denoted $h_2$, is different.    

Given  any three  consecutive intersections with $I$ of the leaf starting at  $P_0$,  there is a  closed geodesic  with respect to $q_2$ homotopic to $\beta$ through the middle point on $I$. 
Thus $C_1$ contains closed geodesics, with respect to the flat structure of $q_2$,  homotopic to $\beta$.
That is, there is a maximal flat cylinder $C_2$ some of whose core curves are contained in $C_1$. 
Since maximal cylinders have singularities on their boundaries, either $P_0$ is on the boundary of $C_2$ or possibly outside it.
It is possible that $Q_0$ is in the interior of $C_2$, so if we take the closed geodesic $\beta$ of $q_2$ through $Q_0$ it does not pass through a singularity.   However in that case, if we similarly take a horizontal segment $I'$ crossing $C_1$ starting at a singularity $P_1$ on the  same boundary component of $C_1$ as $Q_0$, then $\beta$ cuts $I'$ in its interior. This implies that the horizontal distance across $C_2$ is at most $\nu_2(I)+\nu_2(I')\leq 2\epsilon$.  Since the height of $\beta$ is at most $\epsilon_0$, we get the desired area bound for $C_2$. 
Since the horizontal foliations coincide, the lengths of corresponding vertical segments coincide and the lengths of vertical segments crossing the cylinders are comparable.  

\end{proof}

The next Lemma gives a lower bound of extremal length of a curve family in terms of the area of the surface it is contained in and the length of the boundary.

\begin{lem}\label{lower}
 Let $X$ be a Riemann surface. Let $q$ be a unit area quadratic
 differential on $X$.  Let $\Omega$ be a subsurface with geodesic
 boundary and which does not contain a flat cylinder parallel to a boundary component.   If the length $|\partial\Omega|_q$ is small enough, then for any homotopy class  of curves $\alpha \subset \Omega$ with geodesic representative $\alpha$, 
$$\Ext_X(\alpha)\geq \frac{|\alpha|_q^2}{\Area_q(\Omega)+O(|\partial \Omega|_q^2)}.$$
\end{lem}
\begin{proof}
Note that this lemma does not simply follow from the definition of
extremal
length  since the area of $\Omega$ may be smaller than $1$. 
Let $\epsilon=|\partial \Omega|_q$.  Define a metric $\rho$ on $X$ as folows. Let $\rho$ coincide with the $q$-metric on  $\mc N_\e(\Omega)$, the $\epsilon$-neighborhhood of $\Omega$ and  the $q$  metric multiplied by a small $\delta$ on $\Omega'=X\setminus \mc N_\e(\Omega)$. Let $\alpha''$ be any curve in the homotopy class of $\alpha$.
  If $\alpha''$ is not contained in $\Omega$ then 
$\alpha''$ and a segment of $\partial \Omega$ bound a disk. The fact that 
$d_\rho(\Omega',\Omega)=\e$ and $\partial \Omega$ is a geodesic
implies that we can replace an arc of $\alpha''$ with an arc of
$\partial \Omega$ to produce $\alpha'''\subset \overline \Omega$ with smaller length. We conclude that the infimum of the  length in the metric $\rho$ is realized by the geodesic $\alpha$  in $\Omega$. By  definition,
$$\Ext_X(\alpha)\geq \frac{\inf_{\alpha''\sim\alpha}\rho(\alpha'')_q^2}{A(\rho)}\geq \frac{|\alpha|_q^2}{\Area_q(\Omega)+O(\e^2)+\delta\Area_q (\Omega')}.$$ The term $O(\e^2)$ in the inequality above comes from $\Area_q(\mc N_\e(\Omega)\setminus \Omega)$. Since $\delta$ is arbitrary, we have the result.

\end{proof}

\begin{df}
Given a quadratic differential $q$ and $\delta>0$,  a geodesic  $\gamma$ in the $q$ metric is called \textit{almost $(q,\delta)$-vertical} if $v_q(\gamma)\geq \delta h_q(\gamma)$.
\end{df}
Note that $\delta$ may be small in the above definition.  
\begin{lem}
\label{lem:vert}
Let  $q$ be a quadratic differential on $X$ a surface without boundary.   For any $\delta>0$  there is a curve $\be$ which is almost $(q,\delta)$-vertical.

\end{lem}
\begin{proof}
 If $\Gamma_q\neq\emptyset$ there is a vertical saddle connection which is obviously almost $(q,\delta)$-vertical . If a vertical leaf is dense in a subsurface then   the boundary of the subsurface contains a vertical saddle connection. Thus we can assume that the vertical foliation is minimal. Let $A$ be the area of $q$.  The first return map of the foliation to a horizontal transversal $I$ with an endpoint at a singularity defines a generalized interval exchange transformation. 
Choose a horizontal transversal $I$ of length $\lambda$ satisfying 
\begin{equation}
\label{eq:lambda}
\lambda ^2<\frac{A}{\delta}.
\end{equation} 
The transversal $I$ determines a decomposition of the surface into rectangles $\{R_i\}$, with heights $h_i$ and widths $\lambda_i$, whose horizontal sides are subsets of $I$. Each rectangle has two horizontal sides on $I$. Consequently,  if we count each $\lambda_i$ twice  we  have $$\sum_i \lambda_i=2\lambda.$$ 
Since we count each $\lambda_i$ twice we have
$$\sum_i h_i\lambda_i=2A.$$

\begin{figure}[h]
  \hfill
  \begin{minipage}[t]{.45\textwidth}
    \begin{center}  
      \begin{center}

\psfragscanon
\psfrag{3}{$\be_1$}
\psfrag{1}{$p$}
\psfrag{2}{$q$}
\psfrag{5}{$h_j$}
\psfrag{4}{$R_i$}
\psfrag{6}{$\mc N$}
\psfrag{7}{$I$}
\psfrag{9}{$\be_2$}
\includegraphics[]{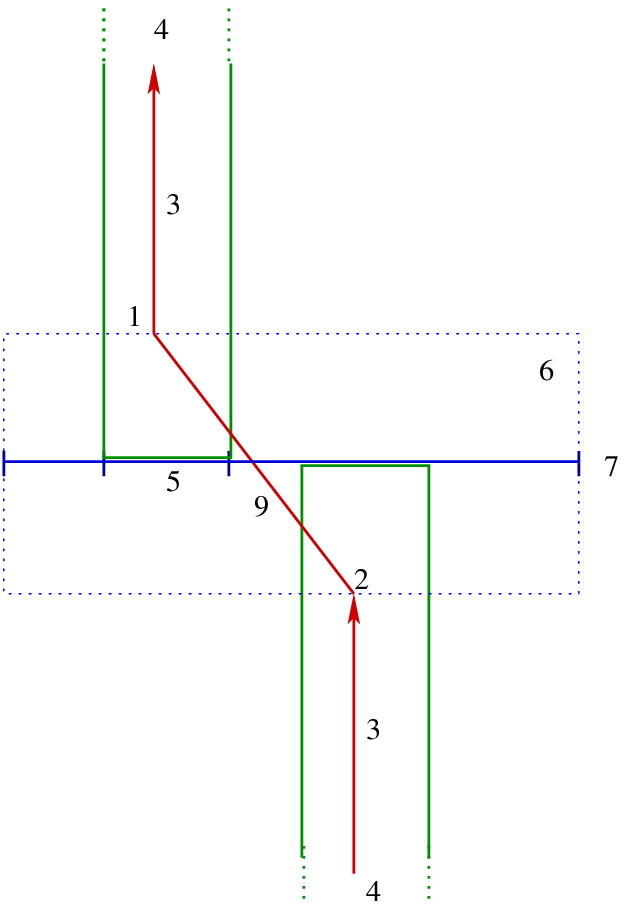}
\caption{ $\be$ is the union of $\be_1$ and $\be_2$}
\label{f1}
\end{center}    \end{center}
  \end{minipage}
  \hfill
  \begin{minipage}[t]{.45\textwidth}
    \begin{center}  
     \psfragscanon
\psfrag{3}{$\be_1$}
\psfrag{1}{$\be_3$}
\psfrag{10}{$\be_2$}
\psfrag{5}{$\be_4$}
\psfrag{4}{$R_i$}
\psfrag{11}{$R_j$}
\psfrag{6}{$\mc N$}
\psfrag{7}{$I^+$}
\psfrag{9}{$I^-$}

\includegraphics[]{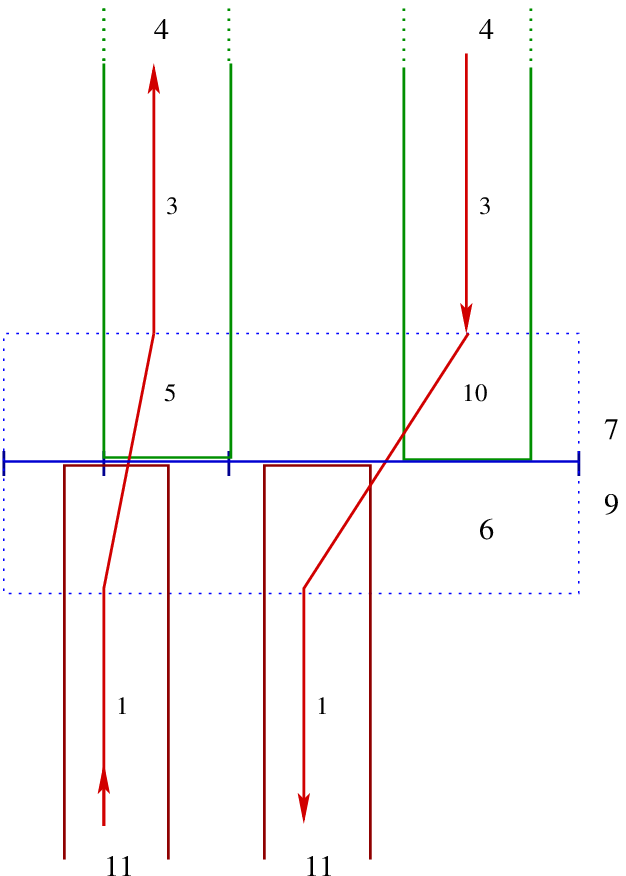}
\caption{$\be$ is the union of $\be_1,\be_2,\be_3$ and $\be_4$}
\label{f2}
    \end{center}
  \end{minipage}
  \hfill
\end{figure}

We conclude that \begin{equation}
\label{eq:max}
\max_ih_i\geq\frac{A}{\lambda}
\end{equation}
 Let $h_i$ realize this maximum. There are two cases. The first case (see Figure \ref{f1}) is that the horizontal sides of $R_i$ are on  opposite sides of $I$. Fix a small neighborhood $\mc{N}$ of $I$. We form a simple closed curve $\be=\be_1*\be_2$. Here $\be_1$ is a vertical segment in $R_i$ whose endpoints $p$ and $q$ are on the boundary of $\mc N$, and $\be_2$ is an arc transverse to the horizontal foliation  in $\mc N$ joining $p$ and $q$. Then $\be$ is also transverse to the horizontal foliation.  Its geodesic representative has the same vertical length as $\be$, namely, $h_i$. The horizontal length of $\beta$ is at most $\lambda$.  Together with (\ref{eq:lambda}) and (\ref{eq:max})  we have that $$\frac{v_q(\beta)}{h_q(\beta)}
\geq \frac{h_i}{\lambda}\geq \frac{A}{\lambda^2}\geq \delta.$$

 In the second case (Figure \ref{f2}), both horizontal sides of $R_i$ are on the same side of $I$ (call it $I^+$). Then there must also be a rectangle $R_j$ with top and bottom on $I_-$. We may form a simple closed curve $\be$ which consists of a vertical segment in $R_i$, a vertical segment in $R_j$ and a pair of arcs in 
 $\mc N$ which are transverse to the horizontal foliation.  Similar to the case above, the ratio of vertical and horizontal components of $\be$ is at least $\delta$.

\end{proof}

\begin{df}
Given a unit area quadratic differential $q$ on a surface $X$ without boundary,    a subsurface
 $\Omega\subsetneq X$ is said to be  $(\epsilon,\epsilon_0)$-thick if the following conditions hold:
\begin{itemize}
\item  $\partial \Omega$ is a geodesic in the metric of $q$.
\item    $\Ext_X(\partial \Omega)\leq \epsilon$
\item the shortest non-peripheral curve in $\Omega$ has $q$ length at least $\epsilon_0$.

\end{itemize}
The surface $X$ itself is $\epsilon_0$-thick if  it satisfies the third condition above. 
\end{df}

The following Lemma says that we can find almost $(q,\delta)$ vertical saddle connections in thick surfaces. 
\begin{lem} \label{vertical} 
For any $B>0,\epsilon_0>0$ there exists $\epsilon>0,\delta>0,D>0$ and $m_0$ such that for any $(\epsilon,\epsilon_0)$-thick subsurface  
$\Omega\subset X$  which does not contain a flat annulus isotopic to a boundary component  and such that $\Area_q(\Omega)\geq B$, the following two conditions hold:  
\begin{enumerate}
\item there is an almost $(q,\delta)$-vertical geodesic  $\gamma$ whose interior lies in $\Omega$ and  
such that  $|\gamma|_q<D$. 
\item For any saddle connection $\gamma$ which is not vertical or horizontal, 
 there is an $m\leq m_0$ and  a collection $\omega_1,\ldots,\omega_m$ of disjoint vertical segments so that for every horizontal leaf $H$ 
$$|\text{card}(H\cap\gamma)-\sum_{i=1}^m \text{card}(H\cap\omega_i)|\leq 2.$$
\end{enumerate}
\end{lem}
\begin{proof}
For the proof of the first statement, we  argue by contradiction. 
Suppose the statement is not true. Then there is a sequence $X_n$ of surfaces, a sequence of unit area quadratic differentials $q_n$ on $X_n$ and 
 a sequence of $(1/n,\epsilon_0)$-thick proper subsurfaces $\Omega_n$ with  $q_n$-area at least $B$  such that 
the shortest almost $(q_n,1/n)$-vertical curve on $\Omega_n$ has length at least $n$. 
We now apply Theorem~\ref{thm:limits} to find  a subsequence $q_n$ which converges uniformly on compact sets to  $q_\infty$ on a limiting surface
$\Omega_\infty$.  The uniform convergence implies that $\Omega_\infty$ is  $\epsilon_0/\sqrt{B}$-thick.

By Lemma~\ref{lem:vert}, taking $\delta=1$, 
 there is a  simple closed curve $\be$ on $\Omega_\infty$ such that 
$$\frac{v_{q_\infty}(\be)}{h_{q_\infty}(\be)}\geq 1.$$
By uniform convergence, $v_{q_n}(\be)\to v_{q_\infty}(\be)$, and  $h_{q_n}(\be)\to h_{q_\infty}(\be)$ and thus for large enough $n$, 
$$\frac{v_{q_n}(\be)}{h_{q_n}(\be)}\geq 1/2$$ and furthermore,
$|\be|_{q_n}\leq |\be|_{q_\infty}+1$. This is a contradiction to the assumption that the shortest $(q_n,1/n)$-vertical curve has length at least $n$, proving the first statement.

We prove the second statement. Begin at one endpont $p$ of $\gamma$
and take the vertical leaf  $\ell_1$ leaving $p$ such that $\gamma$
lies in the $\pi/2$  sector  between $\ell_1$ and a horizontal leaf
leaving $p$.  Move along $\ell_1$  as far as possible to a point $x_1$
in such a way that the segment $\omega_1$ of $\ell_1$, from $p$ to
$x_1$, a horizontal segment $\kappa_1$ from $x_1$ to a point $y_1\in\gamma$ and the
segment $\gamma_1$ of $\gamma$ from $p$  to $y_1$ bounds an embedded
triangle $\Delta_1$ with no singularity in its interior.  If  $\gamma_1=\gamma$
we are done. We take $\omega_1$ as the desired vertical segment.   If
not, then there is a singularity $p_1$ in the interior of
$\kappa_1$. 
At $p_1$ one vertical leaf enters $\Delta_1$. 
 Choose the vertical leaf $\ell_2$ at $p_1$ that makes an angle of
 $\pi$ with the vertical leaf that enters $\Delta_1$ and such that the
 horizontal leaf $\kappa_1$ through $p_1$ on the side of $\Delta_1$ is between them.
Then horizontal leaves through points on $\ell_2$ near $p_1$ will intersect $\gamma$ before returning to $\ell_2$. Now repeat the procedure 
with $\ell_2$ in place of $\ell_1$ and find a maximal   embedded
quadrilateral $\Delta_2$ disjoint from $\Delta_1$  in its interior   consisting  of a pair of horizontal sides, a segment of $\gamma$ and a segment $\omega_2\subset \ell_2$.  We repeat this procedure, if necessary with a new $\ell_3$ until the last segment on $\gamma$ ends at the other endpoint. There are a fixed number of singularities, hence a fixed number of horizontal sides leaving them and so a bounded number of such embedded quadrilaterals.  Say the bound is $m_0$. The desired vertical segments are  $\omega_1,\ldots, \omega_m$, where  $m\leq m_0$.
\end{proof}

For the sequel we will need the following result, due to Rafi, \cite{Ra:1} relating hyperbolic and flat lengths of curves in a thick subsurface. The first statement is Theorem 1, the second is part of Theorem 4 in \cite{Ra:1}
\begin{thm}\label{rafi}
For every  $(\epsilon,\epsilon_0)$-thick subsurface $Y$ of a Riemann surface $X$ with hyperbolic metric $\sigma$ and quadratic differential $q$, there exists $\lambda=\lambda(q,Y)$ such that up to multiplicative constants depending only on topology of $X$
\begin{enumerate}
\item For every non-periferal simple closed curve $\alpha$ in $Y$,   $$|\alpha|_q\asymp \lambda l_{\sigma}(\alpha),$$ the multiplicative constants depending only on the topology of $Y$.
\item $\Area_q(Y)\leq \lambda^2$
\end{enumerate} 
\end{thm}  
  
We will now compare extremal lengths of curves that are contained in the "same" subsurface $\Omega$ measured  with respect to the  metrics defined by two different quadratic differentials $q_1,q_2$ on surfaces $X_1,X_2$.   Specifically, if $\Omega$ is a subsurface with geodesic boundary with respect to $q_1$, and $\Omega$ does not contain a flat cylinder isotopic to a boundary component, then 
we denote by  $\Omega\subset X_2 $ the subsurface containing the same set of simple closed curves and with geodesic boundary with respect to $q_2$. If $\Omega$ is a flat cylinder with respect to $q_1$, then we denote by $\Omega$ the (possibly empty) maximal flat cylinder in the same homotopy class ( with respect to $q_2$.

The following Lemma allows us to find curves with very different extremal length if a subsurface  $\Omega$  has very different areas with respect to two quadratic differentials and one of the surfaces is  thick. 
\begin{lem}\label{area} 
For any $B, M, \delta, \epsilon_0>0$, there exist $\e,C,D>0$  so that the following holds. If $q_1$ and $q_2$ are quadratic differentials on $X_1,X_2$, and $\Omega$ is a proper subsurface with geodesic boundary with respect to  each quadratic differential,  which does not contain a flat cylinder with respect to $q_1$ parallel to a boundary component  and such that $\Omega$ satisfies  
\begin{enumerate}[(i)]
\item \label{a1} $$\Area_{q_1}(\Omega)\geq B,\, \Area_{q_2}(\Omega)<\e$$
\item \label{a2} for any almost $(q_1,\delta)$-vertical curve $\gamma\subset \Omega$  that satisfies $|\gamma|_{q_1}\leq D$, the vertical components satisfy $$\frac{1}{C}\leq\frac{v_{q_1}(\gamma)}{v_{q_2}(\gamma)}\leq C$$
\item \label{a3} $|\partial \Omega|_{q_2}<\e$
\item \label{a4} $\Omega$ is $(\epsilon,\epsilon_0)$- thick with respect to  $q_1$

\end{enumerate}
then there exists a curve $\gamma$ in $\Omega$ so that 
$$\frac{\Ext_{X_2}(\gamma)}{\Ext_{X_1}(\gamma)}\geq M.$$
\end{lem}
\begin{proof}
By Lemma \ref{vertical}, for some $\delta$ and $D$ there is an almost $(q_1,\delta)$-vertical curve $\gamma\subset \Omega$ such that 
\begin{equation}\label{v8}
\epsilon_0\leq |\gamma|_{q_1}<D
\end{equation}

Let $\sigma_i$ be the hyperbolic metric on $X_i$ and $l_{\sigma_i}(\gamma)$ denote the length of the geodesic $\gamma$ in the hyperbolic metric.  
By Theorem \ref{rafi},  
\begin{equation}\label{v1}
l_{\sigma_1}(\gamma)<C_1|\gamma|_{q_1}/\sqrt{\Area_{q_1}(\Omega)} \leq C_1D/\sqrt{B}
\end{equation}
where the constant $C_1$  depends only on the topology of the surface. Also by Maskit's comparison of hyperbolic and extremal lengths, 
\begin{equation}
\Ext_{X_1}(\ga)\leq \frac{1}{2}l_{\sigma_1}(\gamma)e^{l_{\sigma_1}(\gamma)/2}\leq  \frac{1}{2}C_1D/\sqrt{B}e^{C_1D/2\sqrt{B}}.
\end{equation}
Set $C_2= \frac{1}{2}C_1D/\sqrt{B}e^{C_1D/2\sqrt{B}}$, so that 
\begin{equation}\label{v2}
\Ext_{X_1}(\ga)\leq C_2
\end{equation}

On the other hand, by \eqref{v8}, assumption \eqref{a2} and the fact that $\gamma$ is almost $(q_1,\de)$-vertical  
\begin{equation}\label{v3}
|\gamma|_{q_2}\geq v_{q_2}(\ga)>\frac{1}{C}v_{q_1}(\ga)>\frac{\delta}{C(1+\delta)}|\gamma|_{q_1}\geq\frac{\epsilon_0\delta}{C(1+\delta)}.
\end{equation} and by Lemma \ref{lower}, 
\begin{equation}\label{v4}
\Ext_{X_2}(\gamma)\geq\frac{|\gamma|^2_{q_2}}{\Area_{q_2}(\Omega)+O(|\partial \Omega|^2_{q_2})}
\end{equation} Putting the inequalities \eqref{v2}, \eqref{v3}, \eqref{v4}  together and using assumptions \eqref{a1} and \eqref{a3},
 we obtain
\begin{equation}\label{v5}
\frac{\Ext_{X_2}(\gamma)}{\Ext_{X_1}(\gamma)}\geq \frac{\epsilon_0^2\delta^2}{C_2C^2(1+\delta)^2(\Area_{q_2}(\Omega)+O(|\partial \Omega|^2_{q_2}))}\geq \frac{C_3}{\e+O(\e^2)}
\end{equation}
where $C_3=\frac{\epsilon_0^2\delta^2}{C_2C^2(1+\delta)^2}$. Now, setting $\e$ sufficiently small compared to $\frac{C_3}{M}$ guarantees that the Lemma  holds.
\end{proof}

\section{Areas of subsurfaces along rays}

The proof of the main theorem is now based on the next proposition. 
We  have the following set-up. 
 Suppose $q_1,q_2$ are quadratic differentials on $X_1,X_2$ such that the vertical foliations 
$F^v_{q_1},F^v_{q_2}$ are topologically equivalent and have a minimal non uniquely ergodic component $\Omega$.  Suppose also that 
with respect to the  invariant ergodic measures $\nu_1,\ldots, \nu_p$ on $\Omega$, 
$|dx_1|=\sum_{k=1}^pa_k\nu_k$, with  $a_1>0$, while
$|dx_2|=\sum_{k=1}^p b_k\nu_k$ with $b_1=0$. 
Suppose  
$F^h_{q_1}=F^h_{q_2}$. 
Let $|dy|$ denote the transverse measure to this common horizontal
foliation.  We normalize so that \begin{equation}
\label{eq:normalized}
\int_\Omega a_1d\nu_1|dy|=1.
\end{equation}

\begin{prp} 
\label{subsurface}

With the above assumptions, 
 let $X_1(t),X_2(t)$ be the corresponding rays, and let $q_1(t),q_2(t)$ be the quadratic differentials on $ X_1(t), X_2(t)$ respectively.  For any sequence of times $t_n\to\infty$, there is a subsequence, again denoted $t_n$, 
and constants $\epsilon_0>0,c>0$, so that 
 for sufficiently small $\epsilon>0$, there is $t_0$, such that for $t_n\geq t_0$   there
  is a subsurface $Y_1(t_n) \subset\Omega$ satisfying 
\begin{enumerate}[(i)]
\item $Y_1(t_n)$ is $(\epsilon,\epsilon_0)$ thick with respect to $q_1(t_n)$.
\item  $\Area_{q_1(t_n)}(Y_1(t_n))\geq a_1(1-c\epsilon)$.
\item  $\Area_{q_2(t_n)}(Y_1(t_n))<c\epsilon$.
\end{enumerate}
\end{prp}

\begin{proof}[Proof of Proposition \ref{subsurface}]

Since  $F^v_{q_1}$ is minimal and not uniquely ergodic,  
we can apply Theorem 1.1 in \cite{Ma:3}, which says that 
the  ray $X_1(t)$ eventually leaves every compact set in the moduli
space as $t\to \infty$. (That theorem was stated in the case when the minimal component was the entire surface. The proof in the case of a minimal non uniquely ergodic component is identical.  In fact the main idea of the proof is repeated below in a slightly different context). Passing to a subsequence we conclude that there exist 
$\gamma_1(t_n),\ldots, \gamma_m(t_n)\subset\Omega$ such that
  $$\Ext_{X_1(t_n)}(\gamma_i(t_n))\to 0$$ and such that the extremal lengths of all other curves are bounded away from $0$.  

Again, passing to a subsequence,  we can apply Theorem~\ref{thm:limits} to find $\epsilon_0>0$ such that  for $n$ sufficiently large,  there is a nonempty collection $\{Y(t_n)\}$ of disjoint
 $(\epsilon,\epsilon_0)$ thick subsurfaces contained in $\Omega$. We can assume that each $Y(t_n)$ is either a flat annulus or it does not contain a flat annulus isotopic to a boundary component.    There is a uniform bound $N$ for the number of these surfaces. 
Let $f_{t_n}:X_1\to X_1(t_n)$ denote the corresponding Teichmuller map. 

Assume first that $Y(t_n)$ is not a flat cylinder.
Then by passing to a subsequence, we can assume $Y(t_n)$  converges to a limiting punctured surface $Y_\infty$;
the corresponding  $q_1(t_n)$ converges to a limiting $q_{1,\infty}$ on  
$Y_\infty$.
Thus for   any  neighborhood $U$ of the punctures on $Y_\infty$, 
letting $K:=Y_\infty\setminus U$,  
\begin{enumerate}
\item for large enough $n$. there is a conformal map $F_n:K\to Y(t_n)$
\item $F_n^*q_1(t_n)\to q_{1,\infty}$ as $t_n\to\infty$, uniformly on $K$.
\end{enumerate} 

For each such $U$, for $n$ large enough,  the curves $\gamma_i(t_n)$ whose lengths are approaching $0$ satisfy  $$\gamma_i(t_n)\cap F_n(K)=\emptyset.$$ Since $Y(t_n)$
does not contain a flat cylinder in the homotopy class of a component of $\partial Y(t_n)$, we may find $U$ large enough so that
\begin{equation}
\label{eq:small}
\Area_{q_1(t_n)}(Y(t_n)\setminus F_n(K))\leq \epsilon/2.
\end{equation}



Now, since $\nu_i,\nu_j$ are mutually singular measures, 
 there exists $\delta>0$  and  a  finite set $\mc I$ of horizontal transversals $I$ to the vertical foliation in $\Omega$ such that for any $\nu_i,\nu_j, 
i\neq j$ there is a transversal $I_{i,j}\in \mc I$ such that  
\begin{equation}
\label{eq:far}
|\nu_i(I_{i,j})-\nu_j(I_{i,j})|>\delta.
\end{equation} 

 Let $\Lambda_i$ be the set of generic points for $\nu_i$ and the transversals  $\mc I$ ;  that is,   
 $\Lambda_i$ consists of the set of points $x$ such that, if $l_T(x)$ is the vertical leaf segment of $F^v_{q_1}$ through  $x$ of length $T$, then for each $I\in \mc I$
\begin{equation}
\label{eq:converge}
\lim_{T\to \infty}\frac{1}{T} \text{card}(l_T(x)\cap I)=\nu_i(I).
\end{equation}
  The sets $\Lambda_i$ are pairwise disjoint.  With respect to the measure $\nu_i$,  on every transversal  almost every point  belongs to $\Lambda_i$, and, with respect to the area element defined by $q_1$, almost every point   in $\Omega$ belongs to $\cup_{i=1}^p \Lambda_i$.
Let $\Lambda_i(t_n)=f_{t_n}(\Lambda_i)$.

We claim that for $n$ big enough the following holds. Let  $R$ be a coordinate rectangle with respect to the flat structure of $q_{1,\infty}$ (i.e. sides are vertical and horizontal) that is contained in 
$K$.   Then there does {\em not} exist a pair of indices $j\neq i$  and   points $y_{n,i}\in f_{t_n}(\Lambda_i), y_{n,j}\in f_{t_n}(\Lambda_j); i\neq j$ 
such that 
$$z_{n,i}:=F_n^{-1}(y_{n,i})\in R, \ \ z_{n,j}:=F_n^{-1}(y_{n,j})\in R.$$
For suppose there were points with this property. There is a coordinate rectangle $R'\subset R$ whose vertical sides $L_i,L_j$ have endpoints at $z_{n,i},z_{n,j}$.  For every horizontal segment $H$ of $q_{1,\infty}$ 
\begin{equation}\left |\text{card}(H\cap L_i)-\text{card}(H\cap L_j)\right | \leq 2.\nonumber \end{equation}

Let $L_{i,n},L_{j,n}$ be the vertical leaf segments of $q_1(t_n)$ through $y_{i,n},y_{j,n}$ of the same length such that $F_n^{-1}(L_{i,n})$  converges to $L_i$ and similarly with $L_{j,n}$.   As $t_n\to\infty$, since the length of  $I_n=f_{t_n}(I_{i,j})$ goes to infinity, we have  
\begin{equation}\frac{1}{\text{card}(L_{i,n}\cap I_n)}|\text{card}(L_{i,n}\cap I_n)-\text{card}(L_{j,n}\cap I_n))|
\to 0.\nonumber\end{equation}  Mapping $L_{i,n}$ back to $X_1$ by $f_{t_n}^{-1}$, using the fact that $$\frac{\text{card}(f_{t_n}^{-1}(L_{i,n})\cap I_{i,j})}{|f_{t_n}^{-1}(L_{i,n})|}$$ is bounded,  we then have for all large $n$, 
\begin{equation}\frac{1}{|f_{t_n}^{-1}(L_{i,n})|}|\text{card}(f_{t_n}^{-1}(L_{i,n})\cap I_{i,j})-\text{card}(f_{t_n}^{-1}(L_{j,n})\cap I_{i,j})|
\leq\delta/2\nonumber\end{equation} 
and  we have a contradiction to (\ref{eq:far}) and (\ref{eq:converge}).
Thus for each rectangle $R$, there is some  $i=i(R)$ such that for all $j\neq i$ and for all $x\in R$  we have 
\begin{equation}
\label{eq:smallj}
\chi_{F_n(R)}(F_n(x))\chi_{\Lambda_j(t_n)}(F_n(x))\to 0.
\end{equation}

Now we take a covering of $K$ by such rectangles.  If any two rectangles $R,R'$ overlap then $i(R)=i(R')$. It follows from the connectedness of $K$, that there is a single $i$ such that for all $R$, $i(R)=i$.  Thus for $n$ large enough,
for all $j\neq i$, (\ref{eq:smallj}) holds.   
From this it follows that for $n$ large enough,   for $j\neq i$ 
\begin{equation}
\label{eq:disjoint}
\int_{Y(t_n)}d\nu_j|dy|\leq \epsilon
\end{equation} 

We would like to prove an  estimate similar to (\ref{eq:disjoint}) in the case that 
 $Y(t_n)$ is a flat cylinder.  To do that  we need
 a uniform version of generic points.  The reason for that is that there is no natural limiting surface in the case of flat cylinders, and so the previous argument does not quite work.    For each $T_0$, let 
$\Lambda^{T_0}_i$ consist of those $x\in\Lambda_i$ such that for $T\geq T_0$, for each $I\in \mc I$, 
\begin{equation}
\label{eq:T_0}
|\frac{1}{T} \text{card}(l_T(x)\cap I)-\nu_i(I)|<\delta/2.
\end{equation}
Choose   $T_0$ so that with respect to  the measure $d\nu_i$,  except for a set of measure  at most $\epsilon$, every point of $I$ belongs to 
$\Lambda^{T_0}_i$. 

Now   
suppose $Y(t_n)$ is a flat cylinder with core curve $\beta_n$. 
Set $B=a_1(1-N\epsilon)$ and let $\epsilon_0<1$ be a constant such that Lemma~\ref{lem:annulus} holds. 
 Since $f_{t_n}$ is area preserving,  without loss of generality, for $n$ large enough,  we can also assume that at time $t_n$, the core curve  $\beta_n$ is $(q_1(t_n),2)$-almost vertical and has length smaller than  $\epsilon_0$. 
 This means that we can fit coordinate rectangles inside $Y(t_n)$ with
 vertical sides of length at least  $\frac{1}{2}|\beta_n|_{q_1(t_n)}$.    We can choose $n$ large enough so that 
$\frac{1}{2}|\beta_n|_{q_1(t_n)}e^{t_n}\geq T_0$.  This means that if we pull back the vertical segment to $X_1$ its length is at least $T_0$ so that we can apply (\ref{eq:T_0}).  Then the same argument given previously shows that for  all but at most one $j$
$$ f_{t_n}(\Lambda_j^{T_0})\cap Y(t_n)=\emptyset.$$ Otherwise, for a pair $i\neq j$,  we can again find a coordinate rectangle contained in the cylinder  with one vertical side passing through a  point of $f_{t_n}(\Lambda_i^{T_0})$ and the other vertical side passing through a  point of 
$f_{t_n}(\Lambda_j^{T_0})$, and we find $|\nu_i(I)-\nu_j(I)|<\delta$, and again we have a contradiction. Thus we conclude that for $n$ large enough, for all but  one $j$, for any horizontal segment crossing $Y_1(t_n)$ we have \begin{equation}\label{eq:cylinders}
\nu_j(I)<\epsilon.
\end{equation}
Then except for all but at most one index $j$, (\ref{eq:disjoint}) also holds for flat cylinders.

Now    let $Z_1(t_n)$ be the union of those  $Y(t_n)$ such that (\ref{eq:disjoint}) 
holds for the index $j=1$.   Then 
$Z_1(t)\subsetneq \Omega$ for otherwise 
we would have  $$\int_\Omega d\nu_1|dy|<N\epsilon,$$ contradicting 
(\ref{eq:normalized}), for $\epsilon$ sufficiently small.
   Let $Y_1(t_n)=\Omega\setminus Z_1(t_n)$, and so we have 
$$\Area_{q_1(t_n)}(Y_1(t_n))\geq a_1\int_{Y_1(t_n)}d\nu_1|dy|\geq 
a_1(1-N\epsilon)=B.$$ This proves (ii).

We prove (iii). Again the issue is to compare surfaces in different metrics. 
In the case that $Y_1(t_n)$ is not a flat cylinder, 
let $\alpha$ be any  closed geodesic of $q_{1,\infty}$  in $Y_\infty$. 
By (\ref{eq:smallj}), for $j\neq 1$, 
$$\int_{F_n(\alpha)}d\nu_j(t_n)\to 0$$
where $\nu_j(t_n)$ is the push forward measure of $\nu_j$ under $f_{t_n}$. 
But this is then also  true for the geodesic in the class of $F_n(\alpha)$.  
Then (iii) holds by  Lemma~\ref{lem:everyvert}. 
If $Y_1(t_n)$ is a flat annulus we know that for $j\neq 1$, the $\nu_j(t_n)$ measure of any horizontal segment crossing $Y_1(t_n)$ is bounded by $\epsilon$.
We now apply Lemma~\ref{lem:annulus} to give the desired bound on the area.

\end{proof}

\begin{proof}  [Proof of Theorem \ref{main}]
We begin by assuming that the horizontal foliations of $q_1$ and $q_2$ coincide. 
Without loss of generality we can assume 
that  for some minimal component $\Omega$ we have $a_1>0$ and $b_1=0$. 
By (\ref{eq:Kerck}) it suffices to show that for any $M>0$ 
and for any sequence  $t_n\to \infty$,   for $t_n$ sufficiently large 
there is a simple closed curve $\gamma(t_n)$ with 
\begin{equation}\label{big}
\frac{\Ext_{X_2(t_n)}(\gamma(t_n))}{\Ext_{X_1(t_n)}(\gamma(t_n))}>M.
\end{equation}
For all $\e>0$ small  
we now apply  
Proposition~\ref{subsurface}.   We find a fixed constant $B$  such that 
for $t_n$ sufficiently large,  the subsurface $Y(t_n)$, given by that Proposition, 
satisfies 
$$\Area_{q_1(t_n)}(Y(t_n))>B$$ $$\Area_{q_2(t_n)}(Y(t_n))<\e$$ 
and $$\Ext_{X_1(t_n)}(\partial Y(t_n))\leq \epsilon/M.$$
If $|\partial Y(t_n)|_{q_2(t_n)}\geq \sqrt{\epsilon}$, then $\Ext_{X_2(t_n)}
(\partial Y(t_n))\geq \epsilon$, and we are done; we may choose $\gamma(t_n)=\partial Y(t_n)$. 
Thus assume $$|\partial Y(t_n)|_{q_2(t_n)}<\sqrt\epsilon.$$ 

If $Y(t_n)$ is not a flat cylinder,  for $\epsilon$ small enough, we can apply 
Lemma~\ref{vertical} to find a bounded length $(q_1(t_n),\delta)$-almost vertical curve $\gamma_n\subset Y(t_n)$ and then  Lemma~\ref{area}, which says that $\gamma_n$  has the desired property \eqref{big}.   

Thus assume  
  $Y(t_n)$ is a flat cylinder.  Let $\beta_n$ be a core curve of $Y(t_n)$ with $t_n$ chosen so that  $|\beta_n|_{q_1(t_n)}<\e$. Fix some $\delta_0>0$.  Suppose first that $\beta_n$ is $(q_1(t_n),\delta_0)$-almost vertical.  The reciprocal of the modulus of the cylinder is an upper bound for $\Ext_{X_1(t_n)}(\beta_n)$, and we have
\begin{equation}\label {t1}
\Ext_{X_1(t_n)}(\beta_n)\leq \frac{|\beta_n|^2_{q_1(t_n)}}{\Area_{q_1(t_n)}(Y(t_n))}\leq \frac{|\beta_n|^2_{q_1(t_n)}}{B}
\end{equation}
 We now want to estimate $\Ext_{X_2(t_n)}(\beta_n)$. The assumption that $\beta_n$ is $(q_1(t_n),\delta_0)$-almost vertical, since the vertical lengths coincide, implies by  \eqref{v3} that  
$$|\beta_n|_{q_2(t_n)}>\frac{\delta_0|\beta_n|_{q_1(t_n)}}{(1+\delta_0)}.$$ 

Now  there is an annulus $A(t_n)$ which is a union of the flat annulus 
$Y(t_n)$ and an expanding annulus $Y'(t_n)$.
By Proposition~\ref{subsurface} the $q_2(t_n)$-area of $Y(t_n)$ is bounded by $c\epsilon$  for some fixed $c>0$.  The  extremal length of a (homotopy class) curve and the hyperbolic length are asymptotically equal as the quantities go to $0$. Then by Lemma~\ref{lem:ext} there are  
constants $c',c''>0$ depending on $c$ and $\delta_0$ but independent of $\epsilon$  and $t_n$ such that 
$$\Ext_{X_2(t_n)}(\beta_n)\geq c'\ell_{X_2(t_n)}(\beta_n)\geq\frac{c'}{\Mod(Y(t_n))+\Mod(Y'(t_n))}\geq $$
$$\frac{c'}{\frac{\Area_{q_2}(Y(t_n))}{|\beta_n|^2_{q_2(t_n)}}-\log |\beta_n|_{q_2(t_n)}}\geq 
\frac{c''}{\frac{\epsilon}{|\beta_n|_{q_1(t_n)}^2}-\log |\beta_n|_{q_1(t_n)}}.$$
Comparing with  \eqref{t1} we see that for $\epsilon$ small enough, $\beta_n$ is a curve that satisfies \eqref{big}.

Suppose now  the core curve $\beta_n$ of $Y(t_n)$ is not $(q_1(t_n),\delta_0)$ almost vertical.  We can find a smaller time at which the slope of the cylinder is at least $1$. This means we can apply Lemma~\ref{lem:annulus} to find a corresponding cylinder with respect to the metric $q_2(t_n)$. If $Y(t_n)$ is nonseparating, choose a nontrivial  isotopy class of arcs  in the complement of $Y(t_n)$ joining the top and bottom of $Y(t_n)$. If $Y(t_n)$ is separating, choose two nontrivial isotopy classes, one that joins the top of $Y(t_n)$ to itself and the other which joins the bottom to  itself. These families can be chosen to lie in the thick part of the surface $X_1(t_n)$ and  as such have extremal length bounded independently of $t_n$. In the first case we also take a family  of arcs  $\alpha_n$ crossing $Y(t_n)$ that intersect any vertical arc crossing $Y(t_n)$ at most once.  
In the second case we take a pair of (families of) such arcs crossing $Y(t_n)$. These arcs are $\delta$-almost vertical with some uniform constant
$\delta$. 
Now we can form a closed curve $\gamma_n$ as a concatenation of an arc outside $Y(t_n)$ and an arc $\alpha_n$, or, in the separating case, a pair of arcs outside and a pair of arcs crossing. By Theorem~\ref{minsky},  for some constant $c'$   
\begin{equation}
\label{eq:notvertical}
\Ext_{X_1(t_n)}(\gamma_n)\leq c'\Ext_{X_1(t_n)}(\alpha_n)= c'\frac{\inf|\alpha_n|^2_{q_1(t_n)}}{\Area_{q_1(t_n)}(Y(t_n))}\leq 
c'\frac{\inf|\alpha_n|^2_{q_1(t_n)}}{B}
\end{equation}
We consider the corresponding arcs $\alpha_n$ crossing the cylinder with respect to $q_2(t_n)$; defined so that they   intersect the vertical arcs crossing the cylinder at most once. (These arcs may intersect the arcs perpendicular to the core curve $\beta_n$ many times). Their lengths are comparable to the lengths with respect to the metric $q_1(t_n)$.  The corresponding curves $\gamma_n$ formed this way are longer than the arcs $\alpha_n$ crossing the cylinder.  We have for some constant $\delta'$ depending on $\delta$,   
$$\Ext_{X_2(t_n)}(\gamma_n)\geq\Ext_{X_2(t_n)}(\alpha_n)
=\frac{\inf|\alpha_n|^2_{q_2(t_n)}}{\Area_{q_2(t_n)}(Y(t_n))}\geq 
\frac{\inf\delta'|\alpha_n|^2_{q_1(t_n)}}{c\epsilon}.$$ Comparing to (\ref{eq:notvertical}) 
we are done for $\epsilon$ small enough. We have proven the theorem in the case that the horizontal foliations coincide. 

Now consider the 
 general case where the horizontal foliations of $q_1$ and $q_2$ are distinct. Pick a quadratic differential $q_3$ with the same vertical foliation as $q_1$ and the same horizontal foliation as $q_2$.    The rays determined by $q_3$ and $q_2$ diverge by what was already proved. The rays determined by $q_1$ and $q_3$ stay bounded distance apart  as a special case of Ivanov's result \cite{I}.

\end{proof}


\begin{proof}  [Proof of Theorem \ref{top}]
Denote the foliations simply  by $F_1,F_2$.
The first case  is if the minimal components, if any, coincide.  
Since the foliations are not topologically equivalent, and yet have $0$ intersection number, there must be some curve 
$\beta$ which is a core curve of a flat cylinder with respect to  one quadratic
differential, say $q_1$, but is not the core curve of a flat cylinder
of  $q_2$.  Since $\beta$  is isotopic to the core curve of a flat annulus of $q_1$ we have 
  $$\Ext_{X_1(t)}(\beta)\leq  ce^{-2t},$$ for some $c$.
Since $\beta$ is not a subset of a minimal component of
$F_2$, we must have  
$\beta\subset \Gamma_{q_2}$, the critical graph of $q_2$.  Now the length of $\beta$ in the metric of $q_2(t)$ satisfies 
$|\beta|_{q_2(t)}=|\beta|_{q_2}e^{-t}\to 0$.
Now $\beta$ determines an expanding annulus. We  apply the upper  bound for the modulus of that annulus as given in 
Lemma~\ref{lem:ext} and hence the lower bound for extremal length to say that $$\frac{\Ext_{X_2(t)}(\beta)}{\Ext_{X_1(t)}(\beta)}\to\infty.$$

 The second 
case is if one of the foliations, say  $F_1$, has a minimal component
$\Omega_1$ 
which is not a minimal component of $F_2$. Since $i(F_1,F_2)=0$,
every curve $\beta\subset \Omega_1$ satisfies
$i(F_2,\beta)=0$, so that  $\beta\subset \Gamma_{q_2}$, the
critical graph of $q_2$. 
Since $\beta\subset\Omega_1$, we have  $h_{q_1}(\beta)>0$ and so the  flat
length  of $\beta$ with respect to $q_t$ satisfies
$$|\beta|_{q_1(t)}\geq h_{q_1}(\beta)e^t.$$
This gives $$\Ext_{X_1(t)}(\beta)\geq h_{q_1}^2(\beta)e^{2t}.$$
It suffices to find an upper bound for $\Ext_{X_2(t)}(\beta)$. 
Now  $\beta$ is either on the boundary of a minimal component $\Omega_2$ of $F_2$ or is on the boundary of a flat cylinder.   In either case it determines a maximal  expanding annulus $A$.    Since $\Omega_2$ is minimal, the shortest saddle connection $\gamma(t)$ contained in $\Omega_2$  satisfies $|\gamma(t)|_{q_2}\to\infty$ as $t\to\infty$ and hence 
$$e^t|\gamma|_{q_2(t)}\to \infty.$$  This means that for a constant $c>0$, 
$d(A)\geq ce^{-t}$.  Since  $|\beta|_{q_2(t)}\asymp e^{-t}$, by  
 Lemma~\ref{lem:ext} the modulus $A$  is bounded below and so the extremal length of $\beta$ is bounded above. 
\end{proof}


\begin{proof}[Proof of Theorem \ref{approx}]
We note that each $\gamma^j_n$ may itself be a multicurve. 
Fix a finite set of curves $\alpha_1\ldots, \alpha_N$ such that the intersection of any measured foliaiton with these curves determines the foliation. 
Choose a  unit area quadratic differential $q$ on some surface $X$ whose vertical foliation is $[F,\sum_{i=1}^p\nu_i]$. Denote by $|dy|$ the measure on the corresponding horizontal foliation.  Let $X(t)$ be the corresponding ray.    For any sequence of times $t_n\to\infty$ by Proposition~\ref{subsurface} there is $B>0$ and a collection of disjoint domains $Y_1(t_n),\ldots, Y_p(t_n)$ such that the area of $Y_i(t_n)$ with respect to the measure $d\nu_i|dy|$ is at least $B$.   Suppose first that $Y_i(t_n)$ is not a cylinder. By the first part of Lemma~\ref{vertical} we may pick a 
$(q(t_n),\delta)$  almost vertical curve $\gamma_i(t_n)\subset Y_i(t_n)$  of  length at most $D$.   We claim that $\ga_i(t_n)\to [F,\nu_i]$. As before, let $\Lambda_i(t_n)$ be the image of the  generic points inside $Y_i(t_n)$; generic  with respect to the transversals for the set of $\alpha_i$.  The generic points are dense, and $\gamma_i(t_n)$ is a union of a bounded number of saddle connections, so  we can find a bounded collection $\{\omega_j(t_n)\}_{j=1}^m$ of vertical segments beginning at generic points satisfying the second conclusion of  Lemma~\ref{vertical}.  By construction
 of the $\omega_j(t_n)$,  for any fixed $\alpha_k$,
 \begin{equation}
\label{eq:approx}
\frac{i(\gamma_i(t_n),\al_k)}{\sum_{j=1}^m
   i(\omega_j(t_n),\al_k)}\to 1.
\end{equation} Since $\omega_j(t_n)$ is a vertical
 segment through a generic point, 
 $$\frac{\text{card}(\omega_j(t_n)\cap \al_k)}{|\omega_j(t_n)|_{q(0)}}\to \nu_i(\alpha_k).$$ Summing  over all $1\leq j\leq m$ we have
$$\frac{i(\gamma_i(t_n),\al_k)}{\sum_{j=1}^m|\omega_j(n)|_{q(0)}}\to \nu_i(\alpha_k).$$
However $$\frac{\sum_{j=1}^m|\omega_j(n)|_{q(0)}}{v_{q(0)}(\gamma_i(t_n))}\to 1,$$ and so if we let  $s_n=\frac{1}{v_{q(0)}(\gamma_i(t_n))}$
then we have for each $k$, $$\lim_{n\to\infty}s_ni(\gamma_i(t_n),\alpha_k)\to \nu_i(\alpha_k)$$ and we are done.

Finally suppose $Y_i(t_n)$ is a flat cylinder with core curve $\gamma_i(t_n)$.  As in the proof of Proposition~\ref{subsurface} we can assume that $t_n$ is chosen so that $\gamma_i(t_n)$ is  $(q(t_n),2)$-almost vertical. As in that argument we find a dense set of generic points $\Lambda_i^{T_0}$, generic for 
the transversals to the $\alpha_i$. We then can find vertical segments $\omega_j(t_n)$ through generic points such that (\ref{eq:approx}) holds and the rest of the proof is the same.  
\end{proof}


\noindent
Anna Lenzhen:\\
Dept. of Mathematics\\
University of Michigan\\
Ann Arbor, Michigan, 48109\\
E-mail: alenzhen@umich.edu
\medskip

\noindent
Howard Masur:\\
Dept. of Mathematics, University of Chicago\\
Chicago, IL 60637\\
E-mail: masur@math.uchicago.edu

\end{document}